\documentclass[12pt]{article}

\usepackage{forest}
\usepackage{amsthm, amssymb}
\usepackage{amsmath}               
\usepackage{amsfonts}              
\usepackage{amsthm}                
\usepackage{mathrsfs}

\usepackage{verbatim} 
\usepackage{xspace}
\usepackage{graphicx}  
\usepackage[small]{caption} 
\usepackage{color}
\newcommand{\breakingcomma}{%
  \begingroup\lccode`~=`,
  \lowercase{\endgroup\expandafter\def\expandafter~\expandafter{~\penalty0 }}}
\usepackage{float}
\usepackage{wrapfig}

\usepackage{titlesec}
\titleformat{\section}{\large\bfseries}{\thesection}{.5em}{}

\newtheorem{theorem}{Theorem}[section]
\newtheorem{lemma}[theorem]{Lemma}
\newtheorem{remark}[theorem]{Remark}
\newtheorem{proposition}[theorem]{Proposition}
\newtheorem{corollary}[theorem]{Corollary}
\newtheorem{definition}[theorem]{Definition}
\newtheorem{example}[theorem]{Example}
\newcount\refno
\newcommand\be{\begin{equation}}
\newcommand\ee{\end{equation}}
\newcommand\bn{\begin{eqnarray}}
\newcommand\en{\end{eqnarray}}
\newcommand\bns{\begin{eqnarray*}}
\newcommand\ens{\end{eqnarray*}}
\newcommand\bd{\begin{definition}}
\newcommand\ed{\end{definition}}
\newcommand\br{\begin{remark}}
\newcommand\er{\end{remark}}
\newcommand\bt{\begin{theorem}}
\newcommand\et{\end{theorem}}
\newcommand\bp{\begin{proposition}}
\newcommand\ep{\end{proposition}}
\newcommand\bc{\begin{corollary}}
\newcommand\ec{\end{corollary}}
\newcommand\bl{\begin{lemma}}
\newcommand\el{\end{lemma}}

\newcommand\bR{{\mathbb R}}
\newcommand\bN{{\mathbb N}}

\newcommand\bZ{{\mathbb Z}}

\newcommand\cR{{\cal R}}

\newcommand{\N}{\mbox{$\mathbb{N}$}}
\newcommand{\NS}{\mbox{\scriptsize${\mathbb{N}}$}}

\newcommand{\F}{\mbox{$\mathcal{F}$}}

\allowdisplaybreaks

\begin{document}

\title{The Multiple Riordan Group and the Multiple Riordan Type Arrrays}

\author{Tian-Xiao He 
\\
{\small Department of Mathematics}\\
 {\small Illinois Wesleyan University}\\
 {\small Bloomington, IL 61702-2900, USA}\\
}

\date{}

\maketitle
\setcounter{page}{1}
\pagestyle{myheadings}
\markboth{T. X. He }
{The Multiple Riordan Group and the Multiple Riordan Type Arrays}

\begin{abstract}
\noindent 
This is the first paper of a sequence papers on the multiple Riordan group and the multiple Riordon type arrays. We give a comprehensive discussion of the multiple Riordan arrays and characterize them by an $A$-sequence and multiple $Z$-sequences. The multiple Riordan group and some of its subgroups are defined. In addition, we give compressions of multiple Riordan arrays and their sequence characterizations. The total positivity of the compressions of multiple Riordan arrays is studied.

\vskip .2in
\noindent
AMS Subject Classification: 05A15, 05A05, 11B39, 11B73, 15B36, 15A06, 05A19, 11B83.

\vskip .2in
\noindent
{\bf Key Words and Phrases:} multiple Riordan arrays, the multiple Riordan group, generating function, sequence characterization, production matrix, compressions of multiple Riordan arrays, and total positivity.

\end{abstract}

\section{Introduction}

Riordan arrays are infinite, lower triangular matrices defined by the generating function of their columns. They form a group, denoted by $\mathcal{R}$ and called {\em the Riordan group} (see Shapiro, Getu, W. J. Woan and L. Woodson \cite{SGWW}). 

More precisely, let us consider the set of all formal power series (f.p.s.) in $t$, $\F = {\mathbb K}[\![$$t$$]\!]$, with a field ${\mathbb K}$ of characteristic $0$ (e.g., ${\mathbb Q}$, ${\mathbb R}$, ${\mathbb C}$, etc.). The \emph{order} of $f(t)  \in \F$, $f(t) =\sum_{k=0}^\infty f_kt^k$ ($f_k\in {\bR}$), is the minimal number $r\in\N$ such that $f_r \neq 0$. Denote by $\F_r$ the set of formal power series of order $r$. Let $g(t) \in \F_0$ and $f(t) \in \F_1$. 

Similarly, we may consider the set of all formal power series (f.p.s.) in $t^\ell$, $\ell \geq 2$,$\F^{(\ell)} = {\mathbb K}[\![$$t^\ell$$]\!]$, with a field ${\mathbb K}$ of characteristic $0$ (e.g., ${\mathbb Q}$, ${\mathbb R}$, ${\mathbb C}$, etc.). The \emph{order} of $f(t)  \in \F^{(\ell)}$, $f(t) =\sum_{k=0}^\infty f_kt^{\ell k}$ ($f_k\in {\bR}$), is the minimal number $r\in\N$ such that $f_r \neq 0$. Denote by $\F^{(\ell)}_r$ the set of formal power series of order $r$. 

Let $g(t) \in \F_0$ and $f(t) \in \F_1$. Then, the pair $(g(t) ,\,f(t) )$ defines the {\em (proper) Riordan array} $D=(d_{n,k})_{0\leq k\leq n}=(g(t), f(t))$ having
  
\begin{equation}\label{Radef}
d_{n,k} = [t^n]g(t) f(t) ^k
\end{equation}
or, in other words, having $g(t) f(t)^k$ as the generating function whose coefficients make-up the entries of column $k$. 

From the {\it fundamental theorem of Riordan arrays} (see \cite{Sha1}), it is immediate to show that the usual row-by-column product of two Riordan arrays is also a Riordan array:
\begin{equation}\label{Proddef}
    (g_1(t) ,\,f_1(t) )  (g_2(t) ,\,f_2(t) ) = (g_1(t) g_2(f_1(t) ),\,f_2(f_1(t) )).
\end{equation}
The Riordan array $I = (1,\,t)$ acts as an identity for this product. Thus, the set of all Riordan arrays forms the Riordan group $\mathcal{R}$.

In the next section, we define Riordan type arrays and extend them to multiple Riordan type arrays, by using multiple Riordan arrays. The collection of all multiple Riordan arrays form a group called the multiple Riordan group. In Section $3$, we study some properties of Riordan type arrays such that the $A$-sequences and the $Z$-sequences of Riordan type arrays. The sequence characterization of multiple Riordan arrays and multiple Riordan type arrays are presented in Section $4$. We discuss the construction of identities by using Riordan type arrays in Section $5$. Finally, we give and study compressions of multiple Riordan arrays and multiple Riordan type arrays in Section $6$. 

\section{Riordan type arrays, multiple Riordan arrays, and the multiple Riordan group}

We now consider Riordan type arrays or square version of Riordan array $(d_{n,k})_{n,k\in \NS}=(g,f)$ with $g,f\in \F_0$, where $g(0)\not= 0$, but $f(0)$ may or may not be zero, and $d_{n,k}$ is defined by $d_{n,k}=[t^n]gf^k$ for all $n,k\in {\mathbb N}$. 
The collection of all Riordan type arrays is denoted by ${\bR}_s$. The motivation to study Riordan type arrays can be found in the construction of identities (cf. Section $5$) and lattice paths with a step $(0,1)$ (cf. the second paper \cite{He25}).

In construction of a Riordan array, one multiplier function h is used to multiply one
column to obtain the next column. Suppose alternating rules are applied to generate an infinite matrix similar to a Riordan array. To consider this case, one may use $\ell$, $\ell \geq 2$ multiplier functions, denoted by $f_1$, $f_2\ldots,$ and $f_\ell$, respectively.  

Let $g\in {\mathbb K}[[t^\ell]]$ with $g(t)=\sum_{k\geq 0} g_{\ell k}t^{\ell k}$ and $f\in t{\mathbb K}[[t^\ell]]$ with $f_j(t)=\sum_{k\geq 0} f_{j, \ell k+1}$ $t^{\ell k+1}$, $j=1,2,\ldots, \ell$. Then the multiple Riordan matrix in terms of  $g(t)$, $f_i(t)$, $i=1,\ldots, \ell$, denoted by $(g; f_{1}, f_{2}, \ldots, f_\ell)$, is defined by the generating function of its columns as 

\[
(d_{n,k})_{n,k\geq 0}=(g, gf_{1}, gf_{1}f_{2}, \ldots, gf_1f_2\cdots f_\ell, gf_{1}^{2}f_{2}\cdots f_\ell, 
\ldots),
\]
where 

\[
d_{n,k}=[t^n]g f_1^{\lfloor \frac{k+\ell -1}{\ell}\rfloor}f_2^{\lfloor \frac{k+\ell -2}{\ell}\rfloor}\cdots f_\ell^{\lfloor \frac{k}{\ell}\rfloor}.
\]

Here, we use $\ell$ cases of the {\it first fundamental theorem of multiple Riordan type arrays}:

\be\label{1.6}
(g; f_{1}, f_{2},\ldots, f_\ell)A_j(t)=B_j(t),
\ee
where for $A_j(t)=\sum_{k\geq 0}a_{\ell k+j}t^{\ell k+j}$, $j=0,1,\ldots, \ell-1$,
we have $B_j(t)=g(f_j/h)A(h)$, $j=0-,1,\ldots, \ell-1$, where $h=\sqrt[\ell]{f_{1}\cdots f_{\ell}}$. Based on the fundamental theorem of multiple Riordan type arrays, we may define a multiplication of two multiple Riordan type arrays as 

\begin{align}\label{1.7}
&(g; f_{1},f_{2}, \ldots f_\ell)(d; h_{1},h_{2},\ldots, h_\ell)\nonumber\\
=&\left(g d(h); \frac{f_1}{h}h_1(h), 
\frac{f_2}{h}h_2(h),\ldots, \frac{f_\ell }{h}h_\ell (h)\right),
\end{align}
where $d(t)$ = $\sum _{k=0}^{\infty} d_{\ell k}t^{\ell k}$, $h_{j}(t)$ = $\sum_{k=0}^{\infty} h_{j,\ell k} t^{\ell k +1}$, $j=1,2,\ldots, \ell$.

Hence, the multiple Riordan array $(g; f_1,\ldots, f_\ell)$ is corresponding to the multiple Riordan type array $(g; f_1/t, \ldots, f_\ell/t)$. 

The collection of all multiple Riordan arrays forms the multiple Riordan group under the multiplication defined by \eqref{1.7}, in which $d(t)$ = $\sum _{k=0}^{\infty} d_{\ell k}t^{\ell k}$, $h_{j}(t)$ = $\sum_{k=0}^{\infty} h_{j,\ell k+1} t^{\ell k+1}$, $j=1,2,\ldots, \ell$. The collection of all multiple Riordan arrays forms a group called the multiple Riordan group and denoted by ${\cal M}\cal {R}$. Clearly, the identity of ${\cal M}{\cal R}$ is $(1;t,\ldots, t)$, and the inverse of $(g; f_1,\ldots, f_\ell)$ is 

\be\label{1.8}
(g; f_1,\ldots, f_\ell)^{-1}=\left(\frac{1}{g(\bar h)}; \frac{t\bar h}{f_1(\bar h)}, \frac{t\bar h}{f_2(\bar h)}, \ldots, \frac{t\bar h}{f_\ell(\bar h)}\right),
\ee
where $\bar h$ is the compositional inverse of $h=\sqrt[\ell]{f_1f_2\cdots f_\ell}.$

Similarly, the bivariate generating function of the multiple Riordan group element $(g; f_1, f_2, \ldots f_\ell)$ is given by 

\be\label{1.6-2-2}
\frac{g(1 + yf_1 + \cdots +y^{\ell-1}f_1f_2\cdots f_{\ell-1})}{1-y^\ell f_1f_2\cdots f_\ell}.
\ee
In particular, the row sums and the diagonal sums of the associated matrix have generating
functions

\be\label{1.6-3-2}
\frac{g(1 + f_1 + \cdots +f_1f_2\cdots f_{\ell-1})}{1-f_1f_2\cdots f_\ell}
\ee
and

\be\label{1.6-4}
\frac{g(1 + xf_1 + x2f_1f_2 +\cdots +x^{\ell -1}f_1f_2\cdots f_\ell)}{1-x^\ell f_1f_2\cdots f_\ell},
\ee
respectively.

The study of the case $\ell=2$ for $\ell$-multiple Riordan arrays, called double Riordan arrays, is started from Davenport, Shapiro, and Woodson \cite{DSW12}, followed by the author \cite{He18}, Branch, Davenport, Frankson, Jones, and Thorpe \cite{BDFJT}, Davenport, Frankson, Shapiro, and Woodson \cite{DFSW}, Sun and Sun \cite{SS23}, and Zhang and Zhao \cite{ZZ}, etc. and their references. The cases of $\ell=3$ and $4$ are studied in Barry \cite{Bar24} in a different view.  

\begin{proposition}\label{pro:2.3}
We can identify some subgroups of the multiple Riordan group ${\cal M}{\cal R}$:

\begin{itemize} 
\item the set $\mathcal{A}$ of {\em Appell arrays} is the collection of all multiple Riordan arrays $\{(g; t\ldots, t):g\in \F^{(\ell)}_0\}$ in ${\cal M}{\cR}$, which is a subgroup and called the Appell subgroup of ${\cal M}{\cR}$; 

\item the set $\mathcal{L}$ of {\em Lagrange arrays} is the collection of all multiple Riordan arrays $\{ (1; \,f_1,\ldots, f_\ell):f\in \F^{(\ell)}_1\}$ in ${\cal M}{\cR}$, which is a subgroup and called the Lagrange subgroup of ${\cal M}{\cR}$;

\item the set $\mathcal{D}$ of {\it derivative arrays} is the collection of all multiple Riordan arrays $\{ \left(h'(t), f_1,\ldots, f_\ell \right): f_i\in \F^{(\ell)}_1, i=1,\ldots, \ell\}$ with $h=\sqrt[\ell]{f_1\cdots f_\ell}$  in ${\mathcal M}{\cR}$, which is a subgroup and called the derivative subgroup of ${\mathcal M}{\cal M}{\cR}$.

\item the set $\mathcal{B}_j$ of {\it $j$-th Bell arrays} is the collection of all multiple Riordan arrays $\{ (g; f_1,\ldots, f_\ell): g\in \F^{(\ell)}_0, f_i\in \F^{(k)}_1, i=1,\ldots, \ell\,\, and \,\, f_j=tg\}$ in ${\mathcal M}{\cR}$, which is a subgroup and called the $j$-th Bell subgroup of ${\mathcal M}{\cR}$.
\end{itemize} 
\end{proposition}

\section{Sequence characterization of Riordan type arrays}

\begin{definition}\label{def:3.1}
Let $(d_{n,k})_{n,k\in \NS}=(g,f)$ be a Riordan type array, where $g,f\in \F_0$ with $g(0)\not= 0$. Then $(\tilde d_{n,k})_{n,k\in \NS}=(g,tf)$ is called the Riordan array associated with the Riordan type array, where $g,f\in \F_0$ with $g(0)\not= 0$. If $f_0\not= 0$, $(g,tf)$ is a proper Riordan array.
\end{definition}

It is clear that 

\be\label{3.1}
d_{n,k}=\tilde d_{n+k,k} \quad \mbox{and} \quad \tilde d_{n,k}=d_{n-k,k}.
\ee
Hence, we have the following sequence characterization of Riordan type arrays.

\begin{theorem}\label{thm:3.2}
Let $(d_{n,k})_{n,k\in \NS}=(g,f)$ be a Riordan type array, where $g,f\in \F_0$ with $g(0)\not= 0$, and let $(\tilde d_{n,k})_{n,k\in \NS}=(g,tf)$ is the Riordan array associated with the Riordan type array $(g,f)$, where $g,f\in \F_0$ with $g(0)\not= 0$. Then there exists the $A$-sequence $(a_0, a_1, \ldots)$ for $(g,f)$ such that 

\be\label{3.2}
d_{n,k}=\sum^n_{j=0} a_jd_{n-j,k+j-1},
\ee
or equivalently, the generating function of $A$-sequence $A(t)=\sum_{j \geq 0}a_j t^j$ satisfies 

\begin{align}\label{3.3}
&A(tf)=f, \quad \mbox{or equivalently}\nonumber\\
&A(t)=\frac{t}{\overline{tf}},
\end{align}
where $\overline{tf}$ is the compositional inverse of $tf$.

There also exists the $Z$-sequence $(z_0, z_1, \ldots)$ for $(g,f)$ such that 

\be\label{3.4}
d_{n,0}=\sum^{n-1}_{j=0}z_jd_{n-j-1, j},
\ee
or equivalently, the generating function of $Z$-sequence $Z(t)=\sum_{j\geq 0} z_j t^j$ satisfies 

\be\label{3.5}
g=\frac{d_{0,0}}{1-tZ(tf)}.
\ee
\end{theorem}

\begin{remark}
For a Riordan type array $(g,f)$, which associates a Riordan array $(g,tf)$. From the $A$-sequence of $(g,tf)$, we have 

\[
tf=tA(tf),
\]
which also implies \eqref{3.3}. Similarly, from the $Z$-sequence of $(g,tf)$, we have 

\[
g=\frac{d_{0,0}}{1-tZ(tf)},
\]
i.e., \eqref{3.5}.
\end{remark}

\begin{corollary}\label{cor:3.3}
Let $(d_{n,k})_{n,k\in \NS}=(g,f)$ be a Riordan type array, where $g,f\in \F_0$ with $g(0)\not= 0$.
Then, its $A$-sequence $(a_0, a_1, \ldots)$ satisfies 

\be\label{3.6}
d_{n,k}=\sum^k_{\ell =0}\sum^n_{j=1}a_0^\ell a_j d_{n-j, k+j-\ell -1}.
\ee

The $Z$-sequence $(z_0, z_1, \ldots)$ of $(g,f)$ satisfies 

\be\label{3.7}
d_{n,0}=z_0^n d_{0,0}+\sum^{n-1}_{\ell =1}\sum^{n-\ell}_{j =1}z_0^{\ell -1} z_jd_{n-j-\ell, j}.
\ee
\end{corollary}

\begin{proof}
Since 

\begin{align*}
d_{n,k}=&\sum^n_{j=0} a_jd_{n-j,k+j-1}=a_0d_{n, k-1}+\sum^n_{j=i} a_jd_{n-j,k+j-1}\\
=&a_0\sum^n_{j=0} a_jd_{n-j,k+j-2}+\sum^n_{j=1} a_jd_{n-j,k+j-1}\\
=&a_0^2d_{n,k-2}+a_0\sum^n_{j=1} a_jd_{n-j,k+j-2}+\sum^n_{j=1} a_jd_{n-j,k+j-1}\\
=&\cdots\\
=&a_0^k\sum^n_{j=0} a_jd_{n-j,j-1}+a_0^{k-1}\sum^n_{j=1} a_jd_{n-j,j}+\cdots +a_0\sum^n_{j=1} a_jd_{n-j,k+j-2}+\sum^n_{j=1} a_jd_{n-j,k+j-1},
\end{align*}
which implies \eqref{3.6} by assuming $d_{n,-1}=0$. 
we have 

\[
[t^n]gf^k=\sum^k_{j=0}[t^{n-1}]a_j g a_j f^j=[t^{n-1}]g\sum^k_{j=0} a_j f^j,
\]
which implies \eqref{3.1}.

Similarly, from \eqref{3.4} we have 

\begin{align*}
d_{n,0}=& z_0d_{n-1,0}+\sum^{n-1}_{j=1}z_j d_{n-j-1, j}\\
=& z_0\sum^{n-2}_{j=0}z_j d_{n-j-2, j}+\sum^{n-1}_{j=1}z_j d_{n-j-1, j}\\
=&z_0^2 d_{n-2,0}+z_0\sum^{n-2}_{j=1}z_j d_{n-j-2, j}+\sum^{n-1}_{j=1}z_j d_{n-j-1, j}\\
=&\cdots \\
=&z_0^{n-1}d_{1,0}+z_0^{n-2}\sum^{1}_{j=1}z_j d_{1-j,j}+\cdots +z_0\sum^{n-2}_{j=1}z_j d_{n-j-2, j}+\sum^{n-1}_{j=1}z_j d_{n-j-1, j}\\
=&z_0^n d_{0,0}+z_0^{n-2}\sum^{1}_{j=1} z_j d_{1-j,j}+\cdots +z_0\sum^{n-2}_{j=1}z_j d_{n-j-2, j}+\sum^{n-1}_{j=1}z_j d_{n-j-1, j},
\end{align*}
which gives \eqref{3.7}
\end{proof}

\begin{example}\label{ex:2.1}
Considering the Riordan type array $(1/(1-t), 1/(1-t))$, which is called the square matrix version of Pascal's triangle and has the first few rows as 

\[
\left( \frac{1}{1-t}, \frac{1}{1-t}\right)=\left [ \begin{array}{llllll} 1& 1& 1& 1& 1& \cdots\\
1&2 &3& 4& 5&\cdots\\
1&3& 6&10& 15& \cdots\\
1&4& 10& 20&35&\cdots\\
\vdots &\vdots& \vdots& \vdots&\vdots&\ddots\end{array}\right].
\]
Using \eqref{3.2} yields its $A$-sequence $(1,1,0,\ldots)$, i.e., 

\[
d_{n,k}=d_{n,k-1}+d_{n-1, k}.
\]

From \eqref{3.6} and noting $a_0=1$, $a_1=1$, and $a_i=0$ for $i\geq 2$, we have 

\begin{align*}
d_{n,k}=&\sum^k_{\ell=0} d_{n-1, k-\ell}=d_{n-1,0}+d_{n-1,1}+\cdots d_{n-1,k}.
\end{align*}

From either \eqref{3.4} and \eqref{3.7} and noting the $Z$-sequence is $(1,0,\ldots)$, we obtain 

\[
d_{n,0}=d_{n-1,0}=\cdots =d_{0,0}.
\]

Multiplying the square version of Pascal's triangle by the diagonal matrix $diag(2^{0/2}, 2^{1/2}, 2^{2/2}, \ldots)$ yields the classical Delannoy matrix 

\[
\left( \frac{1}{1-t}, \frac{1+t}{1-t}\right)=\left [ \begin{array}{llllll} 1& 1& 1& 1& 1& \cdots\\
1&3 &5& 7& 9&\cdots\\
1&13& 25&41& 61& \cdots\\
1&25& 63& 129&231&\cdots\\
\vdots &\vdots& \vdots& \vdots&\vdots&\ddots\end{array}\right].
\]
It is well known that the entries of the classical Delannoy matrix $(d_{n,k})_{n,k\geq 0}=(1/(1-t), (1+t)/(1-t))$ satisfies the basic relation:

\[
d_{n,k}=\left\{ \begin{array}{ll} 1 &if\,\, n=0\,\, or k=0,\\
d_{n-1, k}+d_{n-1, k-1}+d_{n,k-1}, &otherwise.\end{array}\right.
\]
The above basic relation gives a sequence characterization of the triangle version of the classical Delannoy matrix $(\tilde d_{n,k})_{n,k\geq 0}=(1/(1-t), t(1+t)/(1-t))$ by using \eqref{3.1}:

\[
\tilde d_{n+k,k}=\tilde d_{n+k-1,k}+\tilde d_{n+k-2,k-1}+\tilde d_{n+k-1, k-1}
\]
for $k\geq 1$. The triangle version of the classical Delannoy matrix is 

\[
\left( \frac{1}{1-t}, \frac{t(1+t)}{1-t}\right)=\left [ \begin{array}{llllll} 1& 0& 0& 0& 0& \cdots\\
1&1 &0& 0& 0&\cdots\\
1&3& 1&0& 0& \cdots\\
1&5& 5& 1&0&\cdots\\
1&7& 13& 7&1&\cdots\\
\vdots &\vdots& \vdots& \vdots&\vdots&\ddots\end{array}\right].
\]
\end{example}

\begin{example}\label{ex:2.2}
Let $s(t)=\frac{1+t-\sqrt{1-6t+t^2}}{4t}$ be the generating function of the little Schr\"oder numbers, 
and let $r(t)=\frac{1-t-\sqrt{1-6t+t^2}}{2t}$ be the generating function of the large Schr\"oder numbers. Then the first few rows of Riordan type array $(s,t)$ is 

\[
\left( s(t), 1\right)=\left [ \begin{array}{llllll} 1& 1& 1& 1& 1& \cdots\\
1&1 &1& 1& 1&\cdots\\
3&3& 3&3& 3& \cdots\\
11&11& 11& 11&11&\cdots\\
\vdots &\vdots& \vdots& \vdots&\vdots&\ddots\end{array}\right],
\]
and the first few rows of Riordan type array $(s,s)$ is 

\[
\left( s(t), s(t) \right)=\left [ \begin{array}{llllll} 1& 1& 1& 1& 1& \cdots\\
1&2 &3& 4& 5&\cdots\\
3&7& 12&18& 25& \cdots\\
11&28& 52& 84&125&\cdots\\
\vdots &\vdots& \vdots& \vdots&\vdots&\ddots\end{array}\right].
\]

From \eqref{3.2}, we obtain that the generating functions of $A$-sequences of $(s,1)$ and $(s,s)$ are $A(t)=t$ and $A(t)=\frac{1-t}{1-2t}$, respectively. Hence, the $A$-sequence of $(s,s)$ is $(1,1,2,4,\ldots)$.

From \eqref{3.5}, we know the generating function of $Z$-sequence of $(s,1)$ satisfies 

\[
s(t)=\frac{1}{1-tZ(t)},
\]
or equivalently, 

\begin{align*}
Z(t)=&\frac{1}{t}\left( 1-\frac{1}{s(t)}\right)\\
=&\frac{1}{t}\left( 1-\frac{4t}{1+t-\sqrt{1-6t+t^2}}\right)\\
=&\frac{(1-3t-\sqrt{1-6t+t^2})}{t(1+t-\sqrt{1-6t+t^2})}\\
=&\frac{1-t-\sqrt{1-6t+t^2}}{2t}=r(t).
\end{align*}
Hence, the $Z$-sequence of $(s,1)$ is $(1,2,6,\ldots)$, and the above equation implies $s(t)(1-tr(t))=1$.

Similarly, from \eqref{3.5}, we know the generating function of $Z$-sequence of $(s,s)$ satisfies 

\[
s(t)=\frac{1}{1-tZ(ts)},
\]
or equivalently, 

\[
Z(ts)=\frac{1}{t}\left( 1-\frac{1}{s}\right)=\frac{1}{t}-\frac{1}{ts}.
\]
Since the compositional inverse of $ts$ is 

\[
\overline{ts}=\frac{2t^2-t}{t-1},
\]
we have 

\[
Z(t)=\frac{t-1}{2t^2-t}-\frac{1}{t}=\frac{1}{1-2t},
\]
which gives $Z$-sequence of $(s,s)$, $(1,2,4,\ldots)$. 
\end{example}

In Shapiro and Song \cite{SS}, a matrix poset is defined as 

\begin{definition}\label{def:3.2} \cite{SS}
Let matrices $A$, $B$ and $T$ have only nonnegative integer entries. Then we write 
$A\leq B$ if $AT = B$ and the equality holds only when $T = I$. The matrix $T$ (or $T_{A\to B}$ if needed) is called the transit matrix from $A$ to $B$. Another, maybe more intuitive, partial order has $A\preceq B$ if $a_{i,j}\leq b_{i,j}$ for all $i,j$.
\end{definition}

From \eqref{3.1} we immediately obtain the following result. 

\begin{proposition}\label{pro:3.7}
Let ${\cal R}'_s$ be the subset of ${\cal R}_s$ with nonnegative integer entries, 
and let ${\cal R}'$ be the subset of the Riordan group ${\cal R}$ with nonnegative integer entries.   Then, ${\cal R}'_s$ has a poset defined by 

\be\label{3.8}
(g,f) \preceq (d,h) \quad iff\,\,(g,tf)\preceq (d, th),
\ee
where $g,f,d,h\in \F_0$.
\end{proposition}

\begin{example}
Let $s(t)$ and $r(t)$ be the generating functions of the little Schr\"oder numbers and the large Schr\"oder numbers shown in Example \ref{ex:2.2}, respectively. Then the first few rows of Riordan type array $(s,s^2)$ is

\[
\left( s(t), s(t)^2 \right)=\left [ \begin{array}{llllll} 1& 1& 1& 1& 1& \cdots\\
1&3 &5& 7& 9&\cdots\\
3&12& 25&42& 63& \cdots\\
11&52& 125& 238&399&\cdots\\
\vdots &\vdots& \vdots& \vdots&\vdots&\ddots\end{array}\right].
\]

Similarly, 

\[
\left( s(t), \frac{s-1}{t} \right)=\left [ \begin{array}{llllll} 1& 1& 1& 1& 1& \cdots\\
1&4 &7& 10& 13&\cdots\\
3&17& 40&72& 113& \cdots\\
11&76& 216& 458&829&\cdots\\
\vdots &\vdots& \vdots& \vdots&\vdots&\ddots\end{array}\right].
\]

\[
\left( s(t), \frac{s(s-1)}{t(2-s)}\right)=\left [ \begin{array}{llllll} 1& 1& 1& 1& 1& \cdots\\
1&6 &11& 16& 21&\cdots\\
3&33& 88&168& 273& \cdots\\
11&178& 620& 1462&2829&\cdots\\
\vdots &\vdots& \vdots& \vdots&\vdots&\ddots\end{array}\right].
\]
and

\[
\left( s(t), r(t) \right)=\left [ \begin{array}{llllll} 1& 1& 1& 1& 1& \cdots\\
1&3 &5& 7& 9&\cdots\\
3&11& 23&39& 59& \cdots\\
11&45& 107& 205&347&\cdots\\
\vdots &\vdots& \vdots& \vdots&\vdots&\ddots\end{array}\right].
\]

Thus, we have a poset diagram for the above Riordan type arrays:

\[
(s,1)\longrightarrow (s,s)\longrightarrow (s,r)\longrightarrow \left(s, \frac{s-1}{t}\right) \longrightarrow \left(s, \frac{s(s-1)}{t(2-s)}\right).
\]
A similar poset diagram related $(s, s^2)$ and similar discussions for Fig. 1 of \cite{SS}  can be obtained. 
\end{example}

\section{Sequence characterizations of multiple Riordan arrays and multiple Riordan type arrays}

We now discuss the sequence characterization of multiple Riordan arrays. Inspired by Branch, Davenport, Frankson, Jones, and Thorpe \cite{BDFJT}, Davenport, Frankson, Shapiro, and Woodson \cite{DFSW}, and \cite{He25}, we consider $D=(g;f_1,f_2,\ldots, f_\ell)$ as 

\begin{align}\label{3.9}
D=&(g, gf_1, g(f_1f_2), \ldots, g(f_1f_2\cdots f_\ell), gf_1(f_1f_2\cdots f_\ell), 
\ldots, g(f_1f_2\cdots f_\ell)^2, \ldots)\nonumber\\
=&(g, 0,\ldots,  g(f_1f_2\cdots f_\ell), 0, \ldots, g(f_1f_2\cdots f_\ell)^2, 0, \ldots)\nonumber\\
&+(0, gf_1, 0, \ldots, gf_1(f_1f_2\cdots f_\ell), 0,\ldots,  gf_1(f_1f_2\cdots f_\ell)^2, 0\ldots)\nonumber\\
&+(0,0, gf_1f_2, 0, \ldots, gf_1f_2(f_1f_2\cdots f_\ell), 0,\ldots,  gf_1f_2(f_1f_2\cdots f_\ell)^2, 0\ldots)+\ldots \nonumber\\
\nonumber\\
=&D_0+D_1+D_2+\cdots +D_\ell.
\end{align}
After omitting zero columns and top zero rows, we denote the remaining $D_1$, $D_2,\ldots$, and $D_\ell$ shown above by $D_1^*$, $D_2^*,\ldots$, and $D_\ell^*$, respectively. Then $D_1^*=(g, f_1f_2\cdots f_\ell)=(d^{(1)}_{n,k})_{n,k\geq 0}$, $D_2^*=(gf_1/t, f_1f_2\cdots f_\ell)=(d^{(2)}_{n,k})_{n,k\geq 0}, \ldots$, and $D_\ell^*=(gf_1\cdots f_{\ell -1}/t^{\ell -1}, f_1f_2\cdots f_\ell)=(d^{(\ell)}_{n,k})_{n,k\geq 0}$ are Riordan arrays. Hence, $(g;f_1,f_2,\ldots, f_\ell)$ has $\ell$ $Z$-sequences, denoted by $Z_0$-, $Z_1, \ldots$, $Z_{\ell-1}$-sequences, and one $A$-sequence in this view, while the view shown in \cite{He24, He25} provides a $Z$-sequnce, and $\ell$ $A$-sequences.

\begin{theorem}\label{thm:new-3.1}
Let $(d_{n,k})_{n, k\geq 0}=(g;f_1,f_2,\ldots, f_\ell)$ be a multiple Riordan array, and let $A(t)=\sum_{k\geq 0}a_kt^{\ell k}$, $Z_i(t)=\sum_{k\geq 0} z_{i,k} t^{\ell k}$, $i=0, 1,2,\ldots, \ell -1$, be the generating functions of $A$-, $Z_0$-, $Z_1$-, $\ldots$, and $Z_{\ell -1}$-sequences, respectively. Then

\begin{align}\label{3.10}
&A(t)=\frac{t^\ell}{\overline {h}^\ell},\\
&Z_0(t)=  \frac{1}{\overline{h}^\ell}\left( 1-\frac{g_0}{g(\overline{h})}\right)\label{3.11}\\
&Z_m(t)=\frac{1}{\overline{h}^\ell}\left( 1-\frac{g_0f_{1,1}f_{2,1}\cdots f_{m,1}\overline{h}^m}{g(\overline{h})f_1(\overline{h})f_2(\overline{h})\cdots f_m(\overline{h})}\right)\label{3.12},
\end{align}
where $m=1,2,\ldots, \ell-1$, and $\overline{h}$ is the compositional inverse of $h=\sqrt[\ell]{f_1f_2\cdots f_\ell}$.
\end{theorem}

\begin{corollary}\label{cor:3.12} 
Let $D=(g;f_1,\ldots f_\ell)\in {\cal M}{\cal R}$. Then $D$ is in ${\cal A}$, i.e., $f_1=\cdots=f_\ell=t$  if and only if the generating functions of its $A$-sequences $A=1$. $D$ is in the Lagrange subgroup of ${\cal M\cal R}$, i.e., $g=1$, if and only if the generating functions of the $Z_0$-sequence and $Z_m$-sequence satisfy:

\[
Z_0(t)=0, \quad Z_m(t)=\frac{1}{\bar h^\ell}\left( 1-\frac{f_{1,1}f_{2,1}\cdots f_{m,1}\overline{h}^m}{f_1(\overline{h})f_2(\overline{h})\cdots f_m(\overline{h})}\right),\quad m=1, \ldots, \ell-1.
\]
\end{corollary}

\begin{example}\label{ex:3.3}
Consider Example $1$ given in Barry \cite{Bar24} for the case $\ell=3$, we may set 

\begin{align*}
&g(t)=\frac{1}{1-t^3},\\
&f_1(t)=\frac{t}{1-t^3},\\
&f_2(t)=t(1+t^3),\\
&f_3(t)=\frac{t}{1+t^3}.
\end{align*}
Thus 

\[
h(t)=\sqrt[3]{f_1f_2f_3}=\frac{t}{\sqrt[3]{1-t^3}}\quad \mbox{and}\quad \overline {h}(t)=\frac{t}{\sqrt[3]{1+t^3}}.
\]
The first few rows of Riordan type array $(g;f_1,f_2,f_3)$ is 

\[
\left(\frac{1}{1-t^3};\frac{t}{1-t^3},t(1+t^3),\frac{t}{1+t^3}\right)=\left [ \begin{array}{llllllllll} 
1& 0& 0& 0& 0&0& 0& 0& 0&  \cdots\\
0&1&0 &0& 0& 0& 0& 0& 0& \cdots\\
0&0& 1&0& 0& 0& 0& 0& 0& \cdots\\
1&0& 0& 1&0& 0& 0& 0& 0& \cdots\\
0&2& 0& 0&1& 0& 0& 0& 0& \cdots\\
0&0& 3& 0&0& 1& 0& 0& 0& \cdots\\
1&0& 0& 2&0& 0& 1& 0& 0& \cdots\\
0&3& 0& 0&3& 0& 0& 1& 0& \cdots\\
0&0& 5& 0&0& 4& 0& 0& 1& \cdots\\
\vdots &\vdots& \vdots& \vdots&\vdots&\vdots &\vdots &\vdots &\vdots &\ddots\end{array}\right].
\]

From Theorem \ref{thm:new-3.1}, 

\begin{align*}
&A(t)=1+t^3,\\
&Z_0(t)=1,\\
&Z_1(t)=1+\frac{1}{1+t^3},\\
&Z_2(t)=1+\frac{2}{1+2t^3}.
\end{align*}
Hence, $A$-sequence, $Z_0$-, $Z_1$-, and $Z_2$-sequences are 

\begin{align*}
&A=(1,0,0,1,0,\ldots), \\
&Z_0=(1,0,\ldots),\\
&Z_1=(2,0,0, -1, 0,0, 1,0,0, -1,\ldots),\\
&Z_2=(3, 0,0, -4, 0,0, 8, \ldots).
\end{align*} 

The production matrix of $(g;f_1,f_2, f_3)$ is 

\[
P=\left [ \begin{array}{llllllllll} 
1& 0& 0& 1& 0&0& 0& 0& 0&  \cdots\\
0& 2& 0& 0& 1& 0& 0& 0& 0& \cdots\\
0& 0& 3& 0& 0& 1& 0& 0& 0& \cdots\\
0& 0& 0& 1&0& 0& 1& 0& 0& \cdots\\
0&-1& 0& 0&1& 0& 0& 1& 0& \cdots\\
0& 0& -4& 0&0& 1& 0& 0& 1& \cdots\\
0& 0& 0& 0&0& 0& 1& 0& 0& \cdots\\
0& 1& 0& 0&0& 0& 0& 1& 0& \cdots\\
0& 0& 8& 0&0& 0& 0& 0& 1& \cdots\\
\vdots &\vdots& \vdots& \vdots&\vdots&\vdots &\vdots &\vdots &\vdots &\ddots\end{array}\right].
\]
\end{example}

\begin{definition}\label{def:3.3}
Let $(d_{n,k})_{n,k\in \NS}=(g;f_1,f_2,\ldots, f_\ell)$, where 

\[
d_{n,k}=[t^n]g f_1^{\lfloor \frac{k+\ell -1}{\ell}\rfloor}f_2^{\lfloor \frac{k+\ell -2}{\ell}\rfloor}\cdots f_\ell^{\lfloor \frac{k}{\ell}\rfloor},
\]
be a multiple Riordan type array, and $g,f_i\in \F^{(\ell)}_0$, $i=1,2,\ldots, \ell$, $g(0)\not= 0$. Then $(\tilde d_{n,k})_{n,k\in \NS}=(g;tf_1, tf_2,\ldots, tf_\ell)$ is called the multiple Riordan array associated with the multiple Riordan type array, where $g,f\in \F^{(\ell)}_0$ with $g(0)\not= 0$. If $f_{i,0}\not= 0$, $i=1,2,\ldots, \ell$, $(g;tf_1, tf_2, \ldots, tf_\ell)$ is a proper multiple Riordan array.
\end{definition}

It is clear that \eqref{3.1} holds for the the multiple Riordan array and its association of the multiple Riordan type array, namely, $d_{n,k}=\tilde d_{n+k,k}$ and $\tilde d_{n,k}=d_{n-k,k}$. Hence, we have the following sequence characterization of multiple Riordan type arrays.

\begin{theorem}\label{thm:3.2-2}
Let $(d_{n,k})_{n,k\in \NS}=(g; f_1,f_2,\ldots, f_\ell)$ be a multiple Riordan type array, where $g,f_i\in \F^{(\ell)}_0$, $i=1,2,\ldots, \ell$, with $g(0)\not= 0$, and let $(\tilde d_{n,k})_{n,k\in \NS}=(g;tf_1, tf_2, \ldots, tf_\ell)$ is the multiple Riordan array associated with the multiple Riordan type array $(g;f_1, f_2, \ldots, f_\ell)$, where $g,f_i\in \F^{(\ell)}_0$, $i=1,2\ldots, \ell$, with $g(0)\not= 0$. Then there exists the $A$-sequence $(a_0, a_1, \ldots)$ for $(g;f_1, f_2, \ldots, f_\ell)$ such that 

\be\label{3.2-2}
d_{n,k}=\sum_{j\geq 0} a_jd_{n-j\ell, k+\ell(j-1)},
\ee
or equivalently, the generating function of $A$-sequence $A(t)=\sum_{j \geq 0}a_j t^{j\ell}$ satisfies 

\begin{align}\label{3.3-2}
&A(th)=h^\ell, \quad \mbox{or equivalently,}\nonumber\\
&A(t)=\frac{t^\ell}{\overline{th}^\ell},
\end{align}
where $h=\sqrt[\ell]{f_1f_2\cdots f_\ell}$, and $\overline{th}$ is the compositional inverse of $th$.

There also exists the $Z_0$-sequence $(z_{0,0}, z_{0,1}, \ldots)$ for $(g;f_1,f_2,\ldots, f_\ell)$ such that 

\be\label{3.4-2}
d_{n,0}=\sum_{j\geq 0}z_{m,j} d_{n-(j+1)\ell,j\ell}.
\ee
or equivalently, the generating function of $Z_0$-sequence $Z_0(t)=\sum_{j\geq 0} z_{0,j} t^{j\ell}$ satisfies 

\be\label{3.5-2}
Z_0(t)=\frac{1}{{\overline{th}}^\ell}\left( 1-\frac{g_0}{g(\overline{th})}\right).
\ee

For $m=1,2,\ldots, \ell-1$, there exist the $Z_m$-sequence $(z_{m,0}, z_{m,1}, \ldots)$ for $(g;f_1,f_2,\ldots, f_\ell)$ such that 

\be\label{3.4-3}
d_{n,m}=\sum_{j\geq 0}z_{m,j} d_{n-(j+1)\ell, m+j\ell}
\ee
or equivalently, the generating function of $Z_m$-sequence $Z_m(t)=\sum_{j\geq 0} z_{m,j} t^{j\ell}$ satisfies 

\be\label{3.5-3}
Z_m(t)=\frac{1}{{\overline{th}}^\ell}\left( 1-\frac{g_0f_{1,1}\cdots f_{m,1} {\overline{th}}^m}{g(\overline{th})f_1(\overline{th})\cdots f_m(\overline{th})}\right),
\ee
where $f_{i,1}=[t]f_i$, $i=1,2,\ldots, m$. 
\end{theorem}

\begin{example}\label{ex:3.4}
The multiple Riordan type array corresponding to the multiple Riordan array shown in Example 
\ref{ex:3.3} has the first few rows as 

\[
\left(\frac{1}{1-t^3};\frac{1}{1-t^3},1+t^3,\frac{1}{1+t^3}\right)=\left [ \begin{array}{llllllllll} 
1& 1& 1& 1& 1&1& 1& 1& 1&  \cdots\\
0&0&0 &0& 0& 0& 0& 0& 0& \cdots\\
0&0& 0&0& 0& 0& 0& 0& 0& \cdots\\
1&2& 3& 2&3& 4& 3& 4& 5& \cdots\\
0&0& 0& 0&0& 0& 0& 0& 0& \cdots\\
0&0& 0& 0&0& 0& 0& 0& 0& \cdots\\
1&3& 5& 3&6& 9& 6& 10& 14& \cdots\\
0&0& 0& 0&0& 0& 0& 0& 0& \cdots\\
0&0& 0& 0&0& 0& 0& 0& 0& \cdots\\
1&4& 7& 4&10& 16& 10& 20& 30& \cdots\\
\vdots &\vdots& \vdots& \vdots&\vdots&\vdots &\vdots &\vdots &\vdots &\ddots\end{array}\right].
\]
\end{example}

\section{Construction of identities using Riordan type arrays}

To obtain a class of so-called ``umbral identity'' by using ``umbral operator'', we 
consider the Euler operator $ \theta_t = t D_t $,
where $ D_t = \frac{d}{d t} $ is the usual derivative with respect to $t$.
For this operator, we have the following Grunert formula
\cite{Grunert} \cite[Formula (4.8)]{Gou78} \cite[p.\ 310]{GKP}

\begin{equation}\label{thetaId} 
 \theta_t^m = \sum_{k=0}^m { m \brace k } t^k D_t^k \, .
\end{equation}
Notice that, by applying this formula,
we can find the generating series of the powers $ n^m $ \cite[Formula (4.10)]{Gou78}:
\begin{equation}\label{seriesPowers}
 \sum_{k\geq0} k^m t^k = \sum_{k=0}^m { m \brace k } k! \frac{t^k}{(1-t)^{k+1}} \, .
\end{equation}

\begin{theorem}\label{thm-Umbral01}\cite{He22, HM}
 For every $ m, n \in {\bN} $, we have the identity
 \begin{equation}\label{Umbral01}
  \sum_{k=0}^n { n \choose k } k^m x^{n-k} =
  \sum_{k=0}^{m\wedge n} { m \brace k } { n \choose k } k! (x+1)^{n-k}
 \end{equation}
 or, equivalently,
 \begin{equation}\label{Umbral02}
  \sum_{k=0}^n { n \choose k } k^m (-1)^{n-k} x^{n-k} =
  \sum_{k=0}^{m\wedge n} { m \brace k } { n \choose k } k! (1-x)^{n-k} \, ,
 \end{equation}
 where $m\wedge n=\min\{m,n\}$.
\end{theorem}

We now use umbral operator in Riordan type arrays to construct a class of identities.

\begin{theorem}\label{ThmRioSum01}
 Given a Riordan matrix $ R = [ r_{n,k} ]_{n,k\geq0} = (g(t),f(t)) $,
 consider the associated Riordan type matrices
 $$
  R^\sharp = [ r_{n,k}^\sharp ]_{n,k\geq0} = (g(t),f(t)+1)
  \qquad\text{and}\qquad
  R^\natural = [ r_{n,k}^\natural ]_{n,k\geq0} = (g(t),f(t)-1) \, .
 $$
 Then, for every $ m, n, s \in {\bN} $, we have the identites
 \begin{align}
  & \sum_{k=0}^n { n \choose k } r_{s,n-k} k^m
  = \sum_{k=0}^m { m \brace k } { n \choose k } k! r_{s,n-k}^\sharp \label{RioSum01} \\
  & \sum_{k=0}^n { n \choose k } (-1)^{n-k} r_{s,n-k} k^m
  = \sum_{k=0}^m { m \brace k } { n \choose k } (-1)^{n-k} k! r_{s,n-k}^\natural, \label{RioSum02} \
 \end{align}
where $s\geq n-k$.
\end{theorem}

\begin{example}
Denote by 

\[
F_\ell(n,1)= \frac{1}{\ell n+1} \binom{\ell n+1}{n},
\]
the Fuss-Catalan numbers. For $\ell=2$, the corresponding Fuss-Catalan numbers  
are  the Catalan numbers.  In general, the Fuss-Catalan numbers are numbers of the form 

\begin{equation}\label{1.4-0}
F_\ell(n,r):=\frac{r}{\ell n+r}\binom{\ell n+r}{n},
\end{equation}
The generating function  $F_m(t)$ for the Fuss-Catalan numbers, $\{ F_\ell(n,1)\}_{n\geq 0}$ is called the generalized binomial series in \cite{GKP}. For $\ell=2$, $F_2(t)=C(t)$, the generating function of the Catalan numbers. From the Lambert's formula for the Taylor expansion of the powers of $F_\ell(t)$ (see \cite{GKP}), we have 

\begin{equation}\label{1.4}
F_\ell^r\equiv F_\ell(t)^r=\sum_{n\geq 0}\frac{r}{\ell n+r}\binom{\ell n+r}{n}t^n
\end{equation}
for all $r\in{\bZ}$. Expression \eqref{1.4} implies the following formula of $F_\ell(t)$:

\begin{equation}\label{1.5}
F_\ell(t)=1+tF^\ell_\ell(t).
\end{equation}
Consider the Riordan array 

\[
(F^p_\ell(t), F_\ell(t)-1)=\left[\frac{p+\ell k}{p+\ell n}{p+\ell n\choose n-k}\right]_{n,k\geq 0}
\]
and the Riordan type array 

\[
(F^p_\ell(t), F_\ell(t))=\left[ \frac{p+k}{\ell n+p+k}{\ell n+p+k\choose n}\right]_{n,k\geq 0},
\]
one may apply \eqref{RioSum01} to obtain identity 

\[
  \sum_{k=0}^n { n \choose k } \frac{p+\ell k}{p+\ell s}{ p+\ell s\choose s-k } k^m =
   \sum_{k=0}^m { m \brace k } { n \choose k } \frac{p+k}{\ell s+p+k}{ \ell s+p+k \choose s} k! \, .
\]
\end{example}

\section{Compressions of multiple Riordan arrays and their total positivity} 

Let $(g;f_1,f_2,\ldots, f_\ell)=(d_{n,k})_{n\geq k\geq 0}\in {\cal M}{\cal R}$. We define its compression $(\hat d_{n,k})_{n\geq k\geq 0}$ as follows:

\be\label{6.1} 
\hat d_{n,k}:=d_{n\ell-(\ell-1)k,k}, \quad n\geq k\geq 0. 
\ee

We now study the structure of the compression of a multiple Riordan array starting from the following theorem. 

\begin{theorem}\label{thm:6.2}
Let $(g;f_1,f_2,\ldots, f_\ell)=(d_{n,k})_{n\geq k\geq 0}$ be a multiple Riordan array with 

\begin{align*}
& g(t)=\sum_{k\geq 0} g_{k}t^{k\ell},\,\,  f_i(t)=\sum_{k\geq 0}f_{i,k}t^{k\ell +1}, \,\, 
\end{align*}
for $i=1,2,\ldots, \ell$, and let its compression array $(\hat d_{n,k})_{n,k\geq 0}$ be defined by \eqref{6.1}. Then we have  

\begin{align}\label{6.1-2} 
&\hat d_{n,0}=[t^n] \hat g(t),\nonumber \\
&\hat d_{n,k}=\begin{cases} [t^n]\hat g (\hat f_1\hat f_2\cdots \hat f_\ell)^{k/\ell }, & \mbox{if $k\equiv 0\,(mod\, \ell)$},\\
[t^n]\hat g \hat f_1\cdots \hat f_m(\hat f_1\cdots \hat f_m)^{(k-m)/\ell}, &\mbox{if $k\equiv m\,(mod\, \ell)$}
\end{cases}
\end{align}
for $k\geq 1$, where 
\begin{align}\label{6.1-3}
&\hat g(t)=\sum_{k\geq 0} g_{k}t^k,\,\, \hat f_i(t)=\sum_{k\geq 0}f_{i,k}t^{k+1}, 
\end{align}
for $i=1,2,\ldots, \ell$.
\end{theorem}

\begin{example}\label{ex:3.6}
As an example, $\hat R=(1/(1-t), t, t/(1-t))$ 
is the compression of the double Riordan array $(1/(1-t^2), t, t/(1-t^2))$ (cf. \cite{He18}).  
The compression of $\left(\frac{1}{1-t^3};\frac{t}{1-t^3},t(1+t^3),\frac{t}{1+t^3}\right)$ shown in Example \ref{ex:3.3} is 

\[
\left(\frac{1}{1-t};\frac{t}{1-t},t(1+t),\frac{t}{1+t}\right)=\left [ \begin{array}{lllllllll} 
1& 0& 0& 0& 0&0& 0&0&  \cdots\\
1&1&0 &0& 0& 0& 0&0& \cdots\\
1&2& 1&0& 0& 0& 0& 0& \cdots\\
1&3& 3& 1&0& 0& 0& 0& \cdots\\
1&4& 5& 2&1& 0& 0& 0& \cdots\\
1&5& 7& 3&3& 1& 0& 0& \cdots\\
1&6& 9& 4&6& 4& 1& 0&\cdots\\
1&7& 11& 5&10& 9& 4& 1&  \cdots\\
\vdots &\vdots& \vdots& \vdots&\vdots&\vdots &\vdots  &\ddots\end{array}\right].
\]
\end{example}

The sequence characterization of the compression of a multiple Riordan array is given in the following theorem. 

\begin{theorem}\label{thm:6.3}
Let $(g;f_1,\ldots, f_\ell)=(d_{n,k})_{n\geq k\geq 0}$ be a multiple Riordan array, and let its compression be defined by $(\hat d_{n,k})_{n\geq k\geq 0}$, where $\hat d_{n,k}$ is shown in \eqref{6.1}. 
Suppose the $A$-, $Z_i$-sequences, $i=0,1,\ldots, \ell-1$ of $(g;f_1,f_2\ldots, f_\ell)$ are 

\[
A=\{ a_{0}, a_{1}, \ldots\}, \quad 
Z_i=\{ z_{i,0}, z_{i,1},\ldots\}, \quad i=0,1,\ldots, \ell-1,
\]
with the generating functions $A(t)=\sum_{k\geq 0}a_{k}t^{\ell k}$ and $Z_i(t)=\sum_{k\geq 0} z_{i,k} t^{\ell k}$, $i=0,1,\ldots, \ell-1$. Then, 

\begin{align}\label{3.10-3}
&A\left( \sqrt[\ell]{\frac{\hat f_1 \hat f_2\cdots \hat f_\ell}{t^{\ell -1}}}\right)=\frac{\hat f_1 \hat f_2\cdots \hat f_\ell}{t^\ell},\\
&Z_0\left(\sqrt[\ell]{\frac{\hat f_1 \hat f_2\cdots f_\ell}{t^{\ell -1}}}\right)=\frac{1}{t} \left( 1-\frac{g_0}{\hat g(t)}\right),\label{3.11-3}\\
&Z_m\left(\sqrt[\ell]{\frac{\hat f_1 \hat f_2\cdots f_\ell}{t^{\ell -1}}}\right)=\frac{1}{t}\left( 1-\frac{g_0f_{1,1}f_{2,1}\cdots f_{m,1}t^{m}}{\hat g \hat f_1\hat f_2\cdots \hat f_m}\right),\quad m=1, 2,\ldots, \ell-1,\label{3.12-3}
\end{align}
where $f_{i,1}=[t]f_i$, $i=1,2,\ldots, m$, or equivalently, 

\begin{align}\label{6.2}
&\hat d_{n,k}=\sum_{j\geq 0} a_j \hat d_{n-\ell+j(\ell -1), k+(j-1)\ell},\quad k\geq \ell,\\
&\hat d_{n,m}=\sum_{j\geq 0}z_{m,j} \hat d_{n-1+j(\ell-1), m+j\ell}, \quad m=0,1,\ldots, \ell-1.\label{6.3}
\end{align}
\end{theorem}

The compression of a multiple Riordan type array $(d; h_1,h_2,\ldots, h_\ell)$, $d(t)=\sum_{k\geq 0} d_kt^{k\ell}$ and $h_i(t)=\sum_{k\geq 0} h_{i,k}t^{k\ell}$ is defined by 

\[
(\hat d; \hat h_1, \hat h_2,\ldots, \hat h_\ell)=(d(t^{1/\ell}); h_1(t^{1/\ell}), h_2(t^{1/\ell}),\ldots, h_\ell(t^{1/\ell})).
\]

For example, the compression of the multiple Riordan type array shown in Example 
\ref{ex:3.4} has the first few rows as 

\[
\left(\frac{1}{1-t};\frac{1}{1-t},1+t,\frac{1}{1+t}\right)=\left [ \begin{array}{llllllllll} 
1& 1& 1& 1& 1&1& 1& 1& 1&  \cdots\\
1&2& 3& 2&3& 4& 3& 4& 5& \cdots\\
1&3& 5& 3&6& 9& 6& 10& 14& \cdots\\
1&4& 7& 4&10& 16& 10& 20& 30& \cdots\\
\vdots &\vdots& \vdots& \vdots&\vdots&\vdots &\vdots &\vdots &\vdots &\ddots\end{array}\right].
\]

Following Karlin \cite{Kar} and Pinkus \cite{Pin}, an infinite matrix is called totally positive (abbreviate, TP), if its minors of all orders are nonnegative. An infinite nonnegative sequence $(a_n)_{n\geq 0}$ is called a P\'olya frequency sequence (abbreviate, PF), if its Toeplitz matrix

\[
\left[a_{i-j}\right]_{i,j\geq 0}=\left[ \begin{array} {lllll} a_0 & & & &  \\
a_1& a_0 & & & \\ a_2 &a_1& a_0& &  \\ a_3& a_2 & a_1& a_0 & \\
\vdots& \vdots &\vdots&\vdots & \ddots \end{array}\right]
\]
is TP. We say that a finite sequence $(a_0, a_1,\ldots, a_n)$ is PF if the corresponding infinite sequence $(a_0, a_1, \ldots, a_n, 0, \ldots)$ is PF. Denote by ${\bN}$ the set of all nonnegative integers. A fundamental characterization for PF sequences is given by Schoenberg et al.\cite{AESW, ASW, Kar}, which states that a sequence $(a_n)_{n\geq 0}$ is PF if and only if its generating function can be written as 

\begin{equation}\label{0}
\sum_{n\geq 0} a_n t^n=C t^ke^{\gamma t}\frac{\Pi_{j\geq 0} (1+\alpha_j t)}{\Pi_{j\geq 0} (1-\beta_j t)},
\end{equation}
where $C>0$, $k\in {\bN}$, $\alpha_j$, $\beta_j$, $\gamma \geq 0$, and $\sum_{j\geq 0}(\alpha_j+\beta_j)<\infty$. In this case, the above generating function is called a P\'olya frequency formal power series. For some relevant results, see, for example, Brenti \cite{Bre} and Pinkus \cite{Pin}. 

\begin{theorem}\label{thm:6.1} 
Let $(\hat g;\hat f_1,\hat f_2,\ldots, \hat f_\ell)$ be the compression of a multiple  Riordan array. If $\hat g$, $\hat f_i$, $i=1,2,\ldots, \ell$, are P\'olya frequency formal power series, then $(\hat g;\hat f_1,\hat f_2,\ldots, \hat f_\ell)$ is totally positive.
\end{theorem}

\medbreak

\end{document}